\documentclass[a4paper]{amsart}

\usepackage{amssymb,amsmath,amsthm}
\usepackage{mathrsfs}
\usepackage{graphicx}
\usepackage{wrapfig}
\usepackage{caption}
\usepackage{subcaption}
\usepackage{enumerate}
\usepackage[all]{xypic}
 \usepackage{url}

\newcommand{\ipc}{\textup{\textbf{IPC}\thinspace}}

\newcommand{\logic}{\textup{\textbf{IKT$\Box$}\thinspace}}

\newcommand{\nim}{\textup{\textbf{n1}\thinspace}}
\newcommand{\nima}{\textup{\textbf{n2}\thinspace}}

\newcommand{\multitop}{\textup{\textbf{t2d}\thinspace}}
\newcommand{\multitopp}{\textup{\textbf{t2}\thinspace}}
\newcommand{\multitoppp}{\textup{\textbf{t3}\thinspace}}

\newcommand{\kax}{\textup{\textit{K}\thinspace}}

\newcommand{\tax}{\textup{\textit{T}\thinspace}}
\newcommand{\rnbox}{\textup{\textit{RN}\thinspace}}

\newcommand{\modpon}{\textup{\textit{MP}\thinspace}}

\title[Multi-topological semantics for intuitionistic modal logic]{Multi-topological semantics \\ for intuitionistic modal logic}
\author{Tomasz Witczak}

\address{Institute of Mathematics\\ University of Silesia\\ Bankowa~14\\ 40-007 Katowice\\ Poland}

\email{tm.witczak@gmail.com}

\date{}

\theoremstyle{Theorem}
\newtheorem{tw}{Theorem}[section]

\theoremstyle{Lemma}

\theoremstyle{Definition}
\newtheorem{df}[tw]{Definition}

\theoremstyle{Remark}

\usepackage{graphicx}

\begin{document}

\maketitle

\begin{abstract}
We present three examples of \textit{multi-topological} semantics for intuitionistic modal logic with one modal operator $\Box$ (which behaves in some sense like necessity). We show that it is possible to treat neighborhood models, introduced earlier, as multi-topological. From the neighborhood point of view, our method is based on differences between properties of minimal and maximal neighborhoods. Also we propose transformation of multi-topological spaces into the neighborhood structures.
\end{abstract}

\section{Introduction}
In \cite{atw} we presented sound and complete neighborhood semantics for intuitionistic modal logic (\textit{i. m. l.}) with one modal operator $\Box$ (that of necessity). Our approach was based on the specific properties of minimal and maximal neighborhoods. This framework led us to the i.m.l. with rule of necessity and two modal axioms (\kax and \tax). Such system has been investigated by Bo\v{z}ic and Do\v{s}en in \cite{bozic} - but in bi-relational setting. We have shown that there is strict correspondence between this setting and our neighborhood semantics. 

As for the topological semantics for i.m.l., it has been investigated by Davoren in \cite{davoren}, \cite{dav} and Davoren et al. in \cite{davo}. That approach is more complicated than ours. First, those authors referred to the bi-relational structures with Fischer-Servi conditions (which are not satisfied in our neighborhood framework). Second, their structure consists of one topological space and specific binary relation between points of this space. Our idea is different. In fact, we do not work with one topological space but with a collection of such spaces. Thus, our frames are not strictly \textit{topological} but rather \textit{multi-topological}. Each space is just like maximal neighborhood. However, there are few problems to solve: for example, how to simulate minimal neighborhoods and how to define valuation in a manner that would guarantee pointwise equivalency of neighborhood models and corresponding multi-topological models. 

Another concept has been developed by Collinson et al. in \cite{coll}. It is based on the notion of \textit{topological p-morphism}. These authors started from the relational structures and they used some methods of category theory.

The main reason to study (multi)-topological semantics for i.m.l. is similar to the justification of topological studies in other non-classical logics. Topology allows us to discuss various interesting properties of frames (depending e.g. on axioms of separation or on the notions of density, compactness etc.). Moreover, sometimes these properties can be characterized by means of specific formulas. In our multi-topological setting we can also consider relationships between topological spaces (like inclusions).  However, in this research we concentrate only on basic features of structures in question. Moreover, we did not obtain topological completeness. Not only we did not prove it directly but also our translations between neighborhood structures (for which we have completeness) and multi-topological spaces (which are defined in three slightly different ways) are one-way. Thus, this paper can be considered as a first step in further studies.

\section{Alphabet and language}

Our basic system is named \logic. It has rather standard syntax (i.e. alphabet and language). We use the following notations:

\begin{enumerate}
\item $PV$ is a fixed denumerable set of propositional variables $p, q, r, s, ...$
\item Logical connectives and operators are $\land$, $\lor$, $\rightarrow$, $\bot$, $\Box$.
\item The only derived connective is $\lnot$ (which means that $\lnot \varphi$ is a shortcut for $\varphi \rightarrow \bot$).
\end{enumerate}

Formulas are generated recursively in a standard manner: if $\varphi$, $\psi$ are \textit{wff's} then also $\varphi \lor \psi$, $\varphi \land \psi$, $\varphi \rightarrow \psi$ and $\Box \varphi$. Semantic interpretation of propositional variables and all the connectives introduced above will be presented in the next section. Attention: $\Leftarrow, \Rightarrow$ and $\Leftrightarrow$ are used only on the level of (classical) meta-language.

\section{Neighborhood semantics}
\label{neigh}
\subsection{The definition of structure}

Neighborhoods for pure intuitionistic logic (i.e. without modalities) have been introduced by Moniri and Maleki in \cite{moniri}. If we speak about classical modal logic, then we should note that Pacuit presented an interesting survey of neighborhood semantics for such systems in \cite{pacuit}. Our basic structure is an intuitionistic neighborhood modal frame (\nima-frame) defined as it follows:

\begin{df}
\label{nimdef}
\nima-frame is an ordered pair $\langle W, \mathcal{N} \rangle$ where:

\begin{enumerate}
\item $W$ is a non-empty set (of worlds, states or points)

\item $\mathcal{N}$ is a function from $W$ into $P(P(W))$ such that:

\begin{enumerate}

\item $w \in \bigcap \mathcal{N}_{w}$

\item $\bigcap \mathcal{N}_{w} \in \mathcal{N}_{w}$

\item $u \in \bigcap \mathcal{N}_{w}$ $\Rightarrow$ $\bigcap \mathcal{N}_{u} \subseteq \bigcap \mathcal{N}_{w}$ ($\rightarrow$-\textit{condition})

\item $X \subseteq \bigcup \mathcal{N}_{w}$ and $\bigcap \mathcal{N}_{w} \subseteq X$ $\Rightarrow$ $X \in \mathcal{N}_{w}$ (\textit{relativized superset axiom})

\item $u \in \bigcap \mathcal{N}_{w} \Rightarrow \bigcup \mathcal{N}_{u} \subseteq \bigcup \mathcal{N}_{w}$ \textit{($\Box$-condition)}

\item $v \in \bigcup \mathcal{N}_{w}$ $\Rightarrow$ $\bigcap \mathcal{N}_{v} \subseteq \bigcup \mathcal{N}_{w}$ \textit{($T$-condition)}

\end{enumerate}

\end{enumerate}

\end{df}

\subsection{Valuation and model}

\begin{df}
Neighborhood \nima-model is a triple $F_{\mathcal{N}} = \langle W, \mathcal{N}, V_{\mathcal{N}} \rangle$, where $\langle W, \mathcal{N} \rangle$ is an \nima-frame and $V_{\mathcal{N}}$ is a function from $PV$ into $P(W)$ satisfying the following condition: if $w \in V_{\mathcal{N}}(q)$ then $\bigcap \mathcal{N}_{w} \subseteq V_{\mathcal{N}}(q)$.
\end{df}

\begin{df}
For every \nima-model $M_{\mathcal{N}} = \langle W, \mathcal{N}, V_{\mathcal{N}} \rangle$, forcing of formulas in a world $w \in W$ is defined inductively:
\end{df}
\begin{enumerate}

\item $w \nVdash \bot$

\item $w \Vdash q$ $\Leftrightarrow$ $w \in V_{\mathcal{N}}(q)$ for any $q \in PV$

\item $w \Vdash \varphi \lor \psi$ $\Leftrightarrow$ $w \Vdash \varphi$ or $w \Vdash \psi$

\item $w \Vdash \varphi \land \psi$ $\Leftrightarrow$ $w \Vdash \varphi$ and $w \Vdash \psi$

\item $w \Vdash \varphi \rightarrow \psi$ $\Leftrightarrow$ $\bigcap \mathcal{N}_{w} \subseteq \{v \in W; v \nVdash \varphi$ or $v \Vdash \psi\}$

\item $w \Vdash \Box \varphi$ $\Leftrightarrow$ $\bigcup \mathcal{N}_{w} \subseteq \{v \in W; v \Vdash \varphi\}$.

\end{enumerate}

As we said, $\lnot \varphi$ is a shortcut for $\varphi \rightarrow \bot$. Thus, $w \Vdash \lnot \varphi$ $\Leftrightarrow$ $\bigcap \mathcal{N}_{w} \subseteq \{v \in W; v \nVdash \varphi\}$.

There is also one technical annotation: sometimes we shall write $X \Vdash \varphi$ where $X$ would be a subset of $W$, in particular - minimal or maximal neighborhood (e.g. $\bigcap \mathcal{N}_{w} \Vdash \varphi$). It would mean that each element of $X$ forces $\varphi$.

As usually, we say that formula $\varphi$ is satisfied in a model $M_{\mathcal{N}} = \langle W, \mathcal{N}, V_{\mathcal{N}} \rangle$ when $w \Vdash \varphi$ for every $w \in W$. It is \textit{true} (tautology) when it is satisfied in each \nima-model.

\section{Neigborhood completeness}

In \cite{atw} we have shown (using slightly different symbols) that \nima-frames are sound and complete semantics for the logic \logic defined as the following set of formulas and rules: $\ipc \cup \{\kax, \tax, \rnbox, \modpon\}$, where:

\begin{enumerate}

\item \ipc is the set of all intuitionistic axiom schemes

\item \kax is the axiom scheme $\Box(\varphi \rightarrow \psi) \rightarrow (\Box \varphi \rightarrow \Box \psi)$

\item \tax is the axiom scheme $\Box \varphi \rightarrow \varphi$

\item \rnbox is the rule of necessity: $\varphi \vdash \Box \varphi$

\item \modpon is \textit{modus ponens}: $\varphi, \varphi \rightarrow \psi \vdash \psi$

\end{enumerate}

Completeness result has been established in two ways. First, directly - by means of prime theories and canonical neighborhood model. Second, indirectly - by the transformation into certain class of bi-relational frames, introduced by Bo\v{z}i\'{c} and Do\v{s}en in \cite{bozic} who proved their completeness. Basically, they used different set of axioms.

\section{Multi-topological frames}
\label{mtf}

\subsection{The definition of structure and model}

In this section we introduce the notion of \textit{multi-topological frame (model)}. Such structure can be roughly described as a collection of topological spaces with one valuation. Each space has its \textit{distinguished open set} which plays crucial role in the proof of translation between neighborhood and multi-topological settings.

\begin{df}
\label{multitop}

\multitop-model with distinguished sets is an ordered triple $M_{t} = \langle W, \mathfrak{W}, V_{t} \rangle$ where:

\begin{enumerate}

\item $W \neq \emptyset$.

\item $\mathfrak{W} = \{ \langle T, \tau, D^{T} \rangle \}$, where $T \subseteq W$, $\tau$ is a topology on $T$ (thus $\langle T, \tau \rangle$ is a topological space) and $D^{T} \in \tau, D^{T} \neq \emptyset$.

\item $W = \bigcup \mathcal{T}$, where $\mathcal{T} = \{ T; T \in \langle T, \tau, D^{T} \rangle \in \mathfrak{W} \}$.

\item $V_{t}$ is a function from $PV$ into $P(W)$ satisfying the following condition: $V_{t}(q) = \bigcup \mathcal{X}$ where $\mathcal{X} \subseteq \{X \subseteq W$; there is $\langle T, \tau, D^{T} \rangle \in \mathfrak{W}$ for which $X \in \tau\}$.

\end{enumerate}

\end{df}

The third condition can be formulated also as follows: for each $w \in W$ there is $\langle T, \tau, D^{T} \rangle \in \mathfrak{W}$ such that $w \in T$. Hence, each point of $W$ is in certain topological space. As for the valuation of complex formulas, it is based on the valuation of propositional variables and defined inductively:

\begin{df}
For every \multitop-model $M_{t} = \langle W, \mathfrak{W}, V_{t} \rangle$, valuation of formulas is defined as such:

\begin{enumerate}

\item $V_{t}(\varphi \land \psi) = V_{t}(\varphi) \cap V_{t}(\psi)$

\item $V_{t}(\varphi \lor \psi) = V_{t}(\varphi) \cup V_{t}(\psi)$

\item $V_{t}(\varphi \rightarrow \psi) = \bigcup_{\tau} Int_{\tau} \left(T \setminus (V_{t}(\varphi) \cap -V_{t}(\psi)) \right)$

\item $V_{t}(\Box \varphi) = \bigcup \mathcal{X}$ where $\mathcal{X} = \{X \subseteq W$ such that $X = D^{T}$ for at least one $\langle T, \tau, D^{T} \rangle$ in $\mathfrak{W}$ such that $T \subseteq V(\varphi)\}$.

\end{enumerate}

\end{df}

A few words of comment should be made. As for the valuation of propositional variables, we assume that $V_t(q)$ is a union of sets which are open at least in one topology. Then we go through all the universes and in each one we take standard topological \textit{truth set} for implication. In the next step we form union of all such truth sets. The last important thing is modality: we check which universes satisfy $\varphi$ (which means that they are wholly contained in $V(\varphi)$) and then we take their distinguished sets. Finally, we prepare union of these sets. 

We say that formula $\varphi$ is true \textit{iff} in each \multitop-model $M_{t} = \langle W, \mathfrak{W}, V_{t} \rangle$ we have $V_{t}(\varphi) = W$.

\section{From neighborhood frames to multi-topological structures}

\subsection{Basic notions}

In this section we show that it is possible to treat neighborhood models as multi-topological. First, let us introduce the notion of $w$-open sets.

\begin{df}
\label{wopen}
We say that set $X \subseteq W$ is $w$-open in \nima-frame \textit{iff} $X \subseteq \bigcup \mathcal{N}_{w}$ and for every $v \in X$ we have $\bigcap \mathcal{N}_{v} \subseteq X$. We define $\mathcal{O}_{w}$ as $\{X \subseteq W; X $ are $w$-open $\}$ and call it $w$-topology.
\end{df}

Let us check that this definition is useful for our needs.

\begin{tw}
\label{topol}
Assume that we have \nima-frame $F_{\mathcal{N}} = \langle W, \mathcal{N} \rangle$. Then $\mathcal{O}_{w}$ is a topological space for every $w \in W$.
\end{tw}

\begin{proof}
Let us check standard properties of well-defined topology.
\begin{enumerate}

\item Take empty set. We can say that $\emptyset \in \mathcal{O}_{w}$ because $\emptyset \subseteq \bigcup \mathcal{N}_{w}$ and there are no any $v$ in $\emptyset$.

\item Consider $\bigcup \mathcal{N}_{w}$. Clearly this set is contained in itself and because of $T$-condition we have that for every $v \in \bigcup \mathcal{N}_{w}$ the second condition holds: $\bigcap \mathcal{N}_{v} \subseteq \bigcup \mathcal{N}_{w}$.

\item Consider $\mathscr{X} \subseteq \mathcal{O}_{w}$. We show that $\bigcap \mathscr{X} \in \mathcal{O}_{w}$. The first condition is simple: every element of $\mathscr{X}$ belongs to $\mathcal{O}_{w}$ so it is contained in $\bigcup \mathcal{N}_{w}$. The same holds of course for intersection of all such elements.

    Now let $v \in \bigcap \mathscr{X}$. By the definition we have that $\bigcap \mathcal{N}_{v} \subseteq X$ for every $X \in \mathscr{X}$. Then $\bigcap \mathcal{N}_{v} \subseteq \bigcap \mathscr{X}$.

\item In the last case we deal with arbitrary unions. Suppose that $\mathscr{X} \subseteq \mathcal{O}_{w}$ and consider $\bigcup \mathscr{X}$. Surely this union is contained in $\bigcup \mathcal{N}_{w}$. Now let us take an arbitrary $v \in \bigcup \mathscr{X}$. We know that $\bigcap \mathcal{N}_{v} \subseteq X$ for some $X \in \mathscr{X}$ (in fact, it holds for every $X$ which contains $v$). Then clearly $\bigcap \mathcal{N}_{v} \subseteq \bigcup \mathscr{X}$.

\end{enumerate}

\end{proof}

\begin{figure}[ht]
\centering
\includegraphics[height=4cm]{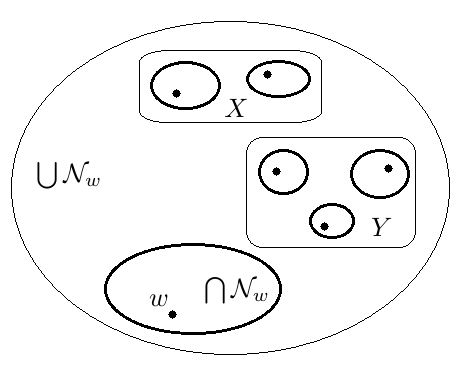}
\caption{Topology ${O}_{w}$. X, Y are $w$-open.}
\label{fig:obrazek {mtop1}}
\end{figure}

One thing should be noted. Clearly, we used $T$-condition to assure that the whole maximal $w$-neighborhood is $w$-open. Basically, in \cite{atw}, we worked with structures without $T$-condition (we may call them \nim-frames). Completeness theorem holds also for them - but it would be at least problematic to treat those frames as multi-topological.

\subsection{Transformation}

\begin{tw}
For each \nima-model $M_{\mathcal{N}} = \langle W, \mathcal{N}, V_{\mathcal{N}} \rangle$ there exists \multitop-model $M_{t} = \langle W, \mathfrak{W}, V_{t} \rangle$ which is \textit{pointwise equivalent} to $M_{\mathcal{N}}$, i.e. $w \Vdash \varphi \Leftrightarrow w \in V_{t} (\varphi)$.
\end{tw}

\begin{proof}
Assume that we have $M_{\mathcal{N}} = \langle W, \mathcal{N}, V_{\mathcal{N}} \rangle$. Now let us consider the following structure: $M_{t} = \langle W, \mathfrak{W}, V_{t} \rangle$ where:

\begin{enumerate}

\item $\mathfrak{W} = \{ \langle \bigcup \mathcal{N}_{w}, \mathcal{O}_{w}, \bigcap \mathcal{N}_{w} \rangle; w \in W \}$

\item for each $q \in PV$, $V_{t}(q) = V_{\mathcal{N}}(q)$

\end{enumerate}

It is easy to check that this is well-defined \multitop-frame. For each $w \in W$ we can treat $\bigcup \mathcal{N}_{w}$ as universe of topological space. Thus $\bigcap \mathcal{N}_{w}$ can be treated as distinguished set in this particular space. Now let us prove pointwise equivalency. Here we use induction by the complexity of formulas.

\begin{enumerate}

\item $\rightarrow$:

($\Rightarrow$)
Suppose that $w \Vdash \varphi \rightarrow \psi$. We want to show there exists certain $\langle \bigcup \mathcal{N}_{x}, \mathcal{O}_{x}, \bigcap \mathcal{N}_{x} \rangle \in \mathfrak{W}$ such that $w \in Int_{x}\left(\bigcup \mathcal{N}_{x} \setminus (V(\varphi) \cap -V(\psi))\right)$.

At first, we can say that $w \in \bigcap \mathcal{N}_{w} \subseteq \{x \in W; x \nVdash \varphi$ or $x \Vdash \psi\}$. But by induction hypothesis, this last set can be written as $\{x \in W; x \notin V_{t}(\varphi)$ or $x \in V_{t}(\psi)\} = - \left(V_{t}(\varphi) \cap -V_{t}(\psi)\right)$. Recall the fact that $ \bigcap \mathcal{N}_{w} \subseteq \bigcup \mathcal{N}_{w}$. Thus $w \in \bigcap \mathcal{N}_{w} \subseteq - \left(V_{t}(\varphi) \cap -V_{t}(\psi)\right) \cap \bigcup \mathcal{N}_{w} = \left(\bigcup \mathcal{N}_{w} \setminus (V_{t}(\varphi) \cap -V_{t}(\psi))\right)$.

But of course $\bigcap \mathcal{N}_{w}$ is $w$-open. Thus it is contained in $Int_{w} \left(\bigcup \mathcal{N}_{w} \setminus (V_{t}(\varphi) \cap -V_{t}(\psi))\right)$. We see that we could treat $w$ as our $x$.

$\Leftarrow$
Now we assume that $w \in V_{t}(\varphi \rightarrow \psi)$. Thus we have certain $\langle \bigcup \mathcal{N}_{x}, \mathcal{O}_{x}, \bigcap \mathcal{N}_{x} \rangle \in \mathfrak{W}$ such that $w \in Int_{x}\left(\bigcup \mathcal{N}_{x} \setminus (V_{t}(\varphi) \cap -V_{t}(\psi))\right)$. By induction hypothesis, we can say that $w \in Int_{x} \left(\bigcup \mathcal{N}_{x} \setminus \{z \in W; z \Vdash \varphi \land z \nVdash \psi\} \right)$.
Hence, $w$ belongs to the biggest $x$-open set $X$ such that $X \subseteq \{z \in W; z \nVdash \varphi \text{ or } z \Vdash \psi\}$.

But if $X$ is $x$-open then (by the definition of topology $\mathcal{O}_{x}$) we can say that $\bigcap \mathcal{N}_{w} \subseteq X$. In particular, $\bigcap \mathcal{N}_{w} \subseteq \{z \in W; z \nVdash \varphi$ or $z \Vdash \psi\}$. Thus $w \Vdash \varphi \rightarrow \psi$.

\item $\Box$:

$\Rightarrow$

Assume that $w \Vdash \Box \varphi$. Thus $\bigcup \mathcal{N}_{w} \subseteq \{x \in W; x \Vdash \varphi\}$. We want to show that $w \in V_{t}(\Box \varphi)$, i.e. that there is $X \subseteq W$ such that $w \in X$ and for certain $\langle \bigcup \mathcal{N}_{x}, \mathcal{O}_{x}, \bigcap \mathcal{N}_{x} \rangle \in \mathfrak{W}$ we have: $X = \bigcap \mathcal{N}_{x}$, $\bigcup \mathcal{N}_{x} \subseteq V_{t}(\varphi)$.

Surely, we can take $x = w$. Now, if $w \Vdash \Box \varphi$, then $\bigcup \mathcal{N}_{w} \subseteq V_{\mathcal{N}}(\varphi)$. By induction hypothesis, $\bigcup \mathcal{N}_{w} \subseteq V_{t}(\varphi)$.

$\Leftarrow$

Suppose that $w \in V_{t}(\Box \varphi)$. Thus $w \in X \subseteq W$ such that for certain $\langle \bigcup \mathcal{N}_{x}, \mathcal{O}_{x}, \bigcap \mathcal{N}_{x} \rangle \in \mathfrak{W}$ we can say that $X = \bigcap \mathcal{N}_{x}$ and $\bigcup \mathcal{N}_{x} \subseteq V_{t}(\varphi)$.

If $\bigcup \mathcal{N}_{s} \subseteq V_{t}(\varphi)$, then - by induction hypothesis - $\bigcup \mathcal{N}_{x} \subseteq V_{\mathcal{N}}(\varphi)$. Thus $x \Vdash \Box \varphi$. But $w \in \bigcap \mathcal{N}_{x}$. Thus, by the monotonicity of intuitionistic forcing, $w \Vdash \Box \varphi$.

\end{enumerate}

\end{proof}

\begin{figure}[ht]
\centering
\includegraphics[height=3.5cm]{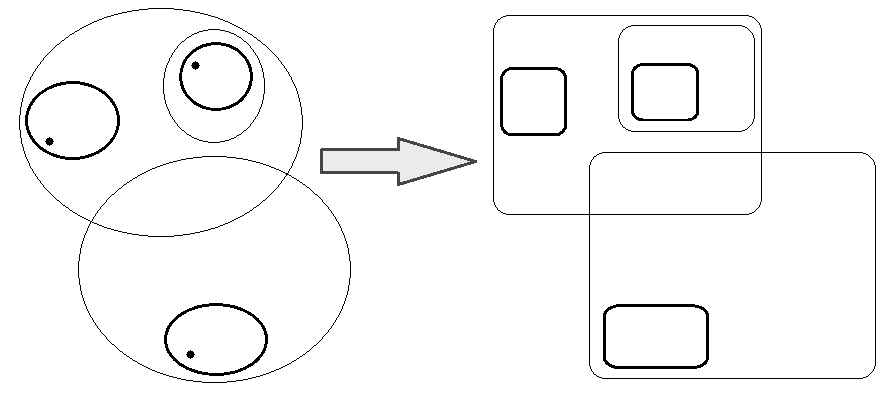}
\caption{From neighborhoods to multi-topological space with distinguished sets.}
\label{fig:obrazek {mtop2}}
\end{figure}

\section{From multi-topological structures to neighborhood structures}
\label{sect}
In the former section we used multi-topological structures with distinguished open sets $D^{T}$. Those sets can be considered as analogues or images of minimal $w$-neighborhoods (while universes of topological spaces played the role of maximal $w$-neighborhoods). We used such unconventional approach mainly because our topology $\mathcal{O}_{w}$ does not "recognize" minimal neighborhoods. Thus, if we have $\bigcup \mathcal{N}_{w}$, then from the neighborhood point of view $\bigcap \mathcal{N}_{w}$ is specific - but as $w$-open set it is not distinguished in any way from other $w$-open sets. But we need such distinction to establish correspondence between $V_{\mathcal{N}}$ and $V_{t}$.

Now we are on the other side: we start from multi-topological structures but defined in slightly different way. Here we do not have $D^{T}$ sets - but our assumptions about topologies are stronger. In the former sections we treated each $\tau$ as an arbitrary topological space, even if it is clear that for each $w \in W$, our $\mathcal{O}_{w}$ must be an Alexandroff space. It means that each intersection of $w$-open sets is open (or - equivalently - that each point has a minimal neighborhood, where the notion of neighborhood is understood in a topological sense).

Now we \textit{start} from Alexandroff spaces. Thus, we have the following definition (of frame):

\begin{df}
\label{multitop}

\multitopp-frame is an ordered pair $\langle W, \mathfrak{W} \rangle$ where:

\begin{enumerate}

\item $W \neq \emptyset$

\item $\mathfrak{W} = \{ \langle T, \tau \rangle \}$, where $T \subseteq U$ and $\tau$ is an Alexandroff topology on $T$.

\item $W = \bigcup \mathcal{T}$, where $\mathcal{T} = \{ T; T \in \langle T, \tau \rangle \in \mathfrak{W} \}$

\end{enumerate}
\end{df}

Each $\langle T, \tau \rangle$ is an Alexandroff space, so each $w \in T$ has its minimal $\tau$-open neighborhood. If we denote the family of $\tau$-open $w$-neighborhoods as $\mathcal{O}^{w}_{\tau}$, then we can introduce the following notation: $\bigcap \mathcal{O}^{w}_{\tau} = \text{min} \mathcal{O}^{w}_{\tau}$.

Now we discuss intersection of all such minimal $\tau$-open $w$-neighborhoods. It will be denoted as $\bigcap_{\langle T, \tau \rangle \in \mathfrak{U}} \{\text{min} \mathcal{O}^{w}_{\tau}\}$. In the next step we show how to define (in this topological environment) certain kind of \textit{neighborhoods} (but in the sense of \nima-frames).

\begin{df}
\label{minmax}
Assume that we have \multitopp-frame $\langle W, \mathfrak{W} \rangle$. Then for each $w \in W$ we define:

\begin{enumerate}
\item $\bigcap \mathcal{N}^{t}_{w} = \bigcap_{\langle T, \tau \rangle \in \mathcal{T}^{w}} \{\text{min} \mathcal{O}^{w}_{\tau}\}$, where $\mathcal{T}^{w} = \{ \langle T, \tau \rangle \in \mathfrak{W}; w \in T \}$.

\item $\bigcup \mathcal{N}^{t}_{w} = \bigcap \mathcal{T}^{w}$.

\item $X \in \mathcal{N}^{t}_{w} \subseteq P(P(W)) \Leftrightarrow \bigcap \mathcal{N}^{t}_{w} \subseteq X \subseteq \bigcup \mathcal{N}^{t}_{w}$.

\end{enumerate}

\end{df}

\begin{tw}
Assume that we have \multitopp-frame $\langle W, \mathfrak{W} \rangle$ with $\mathcal{N}^{t}_{w}$ defined as in Def. \ref{minmax}. We state that for each $w \in U$, $\mathcal{N}^{t}_{w}$ has all the properties of neighborhood family in \nima-frame.
\end{tw}

\begin{proof}
We must check five conditions:

\begin{enumerate}

\item $w \in \bigcap \mathcal{N}^{t}_{w}$. This is simple because $\bigcap \mathcal{N}^{t}_{w}$ is defined as an intersection of all $\tau$-open $w$-neighborhoods (for every $\tau$) and certainly $w$ is in each such neighborhood.

\item $\bigcap \mathcal{N}^{t}_{w} \in \mathcal{N}^{t}_{w}$. This is obvious by the very definition of $\mathcal{N}^{t}_{w}$.

\item $v \in \bigcap \mathcal{N}^{t}_{w} \Rightarrow \bigcap \mathcal{N}^{t}_{v} \subseteq \bigcap \mathcal{N}^{t}_{w}$. Let us note two facts. First, $v$ is at least in all those spaces, in which $w$ is (because it is in the intersection of all minimal $w$-neighborhoods). Thus, we can say that $\bigcap \mathcal{N}^{t}_{v} = \bigcap_{\langle T, \tau \rangle \in \mathcal{T}^{v}} \{\text{min} \mathcal{O}^{v}_{\tau}\} \subseteq \bigcap_{\langle T, \tau \rangle \in \mathcal{T}^{w}} \{\text{min} \mathcal{O}^{v}_{\tau}\}$.

    Second, suppose for a moment that we work with one particular Alexandroff topological space $\rho$. Assume that $v$ belongs to the minimal $\rho$-open neighborhood of $w$. Of course $v$ has its own minimal $\rho$-open neighborhood - but let us suppose that $\text{min} \mathcal{O}^{v}_{\rho} \nsubseteq \text{min} \mathcal{O}^{w}_{\rho}$. Now - from the basic properties of topology and the fact that at least $v$ belongs to $\text{min} \mathcal{O}^{w}_{\rho}$ - we state that $\text{min} \mathcal{O}^{v}_{\rho} \cap \text{min} \mathcal{O}^{w}_{\rho}$ is $\rho$-open. Of course, this intersection is contained in $\text{min} \mathcal{O}^{w}_{\rho}$. Thus, we have contradiction with the assumption that minimal $\rho$-open $v$-neighborhood is not contained in $\text{min} \mathcal{O}^{w}_{\rho}$.

 Now let us go back to the main part of the proof. The second fact allows us to say that $\bigcap_{\langle T, \tau \rangle \in \mathcal{T}^{w}} \{\text{min} \mathcal{O}^{v}_{\tau}\} \subseteq \bigcap_{\langle T, \tau \rangle \in \mathcal{T}^{w}} \{\text{min} \mathcal{O}^{w}_{\tau}\} = \bigcap \mathcal{N}^{t}_{w}$.

\item $v \in \bigcap \mathcal{N}^{t}_{w} \Rightarrow \bigcup \mathcal{N}^{t}_{v} \subseteq \bigcup \mathcal{N}^{t}_{w}$. As earlier, we say that $v$ is at least in each space which belongs to $\mathcal{T}^{w}$. Thus $\bigcup \mathcal{N}^{t}_{v} = \bigcap \mathcal{T}^{v} = \bigcap \{ \langle T, \tau \rangle \in \mathfrak{U}; v \in T \} \subseteq \bigcap \{ \langle T, \tau \rangle \in \mathfrak{U}; w \in T \} = \bigcup \mathcal{N}^{t}_{w}$.

\item $v \in \bigcup \mathcal{N}^{t}_{w} \Rightarrow \bigcap \mathcal{N}^{t}_{v} \subseteq \bigcup \mathcal{N}^{t}_{w}$. Suppose that $v \in \bigcup \mathcal{N}^{t}_{w}$ defined as in Def. \ref{minmax}. Thus $v \in \bigcap \mathcal{T}^{w}$ which means in particular that $v$ is in all those universes, in which $w$ is. Now it is clear that $\bigcap \mathcal{N}^{t}_{v}$ - defined as an intersection of all $\tau$-open minimal $v$-neighborhoods - must be contained at least in each element of $\mathcal{T}^{w}$, i.e. in $\bigcup \mathcal{N}^{t}_{w}$.

\end{enumerate}

\end{proof}

\begin{figure}[ht]
\centering
\includegraphics[height=4cm]{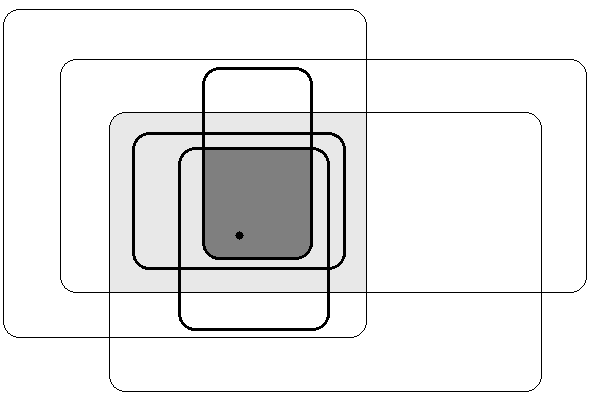}
\caption{Maximal and minimal neighborhoods in multi-topological space.}
\label{fig:obrazek {mtop3}}
\end{figure}

We have transformed our initial multi-topological structure into the neighborhood frame. Note that it is possible that for each $\tau$ the set $\bigcap \mathcal{N}^{t}_{w}$ is \textit{not} $\tau$-open. We do not expect this. It is just intersection of all minimal $w$-neighborhoods. Now we shall introduce valuation and rules of forcing - thus obtaining logical model.

\begin{df}

Assume that we have \multitopp-frame $\langle W, \mathfrak{W} \rangle$. Suppose that for each $w \in W$ we defined $\mathcal{N}^{t}_{w}$ as in Def. \ref{minmax}. We define valuation $V_{t}$ as a function from $PV$ into $P(W)$ satisfying the following condition: if $w \in V_{t}(q)$ then $\bigcap \mathcal{N}^{t}_{w} \subseteq V_{t}(q)$. The whole triple $\langle W, \mathfrak{W}, V_{t} \rangle$ is called \multitopp-model.

\end{df}

\begin{df}
For every \multitopp-model $M_{t} = \langle W, \mathfrak{W}, V_{t} \rangle$, valuation of formulas is defined as such:

\begin{enumerate}

\item $V_{t}(\varphi \land \psi) = V_{t}(\varphi) \cap V_{t}(\psi)$

\item $V_{t}(\varphi \lor \psi) = V_{t}(\varphi) \cup V_{t}(\psi)$

\item $V_{t}(\varphi \rightarrow \psi) = \bigcup_{w \in \mathscr{Imp}} \{\bigcap \mathcal{N}^{t}_{w}\}$ where $\mathscr{Imp} = \{w \in W; \bigcap \mathcal{N}^{t}_{w} \subseteq -V_{t}(\varphi) \cup V_{t}(\psi)\}$

\item $V_{t}(\Box \varphi) = \bigcup_{w \in \mathscr{Mod}} \{\bigcap \mathcal{N}^{t}_{w}\}$ where $\mathscr{Mod} = \{w \in W; \bigcup \mathcal{N}^{t}_{w} \subseteq V_{t}(\varphi)\}$
9
\end{enumerate}

\end{df}

We say that formula $\varphi$ is true \textit{iff} in each \multitopp-model $M_{t} = \langle W, \mathfrak{W}, V_{t} \rangle$ we have $V_{t}(\varphi) = W$.

The next theorem is crucial for our considerations.

\begin{tw}
For each \multitopp-model $M_{t} = \langle W, \mathfrak{W}, V_{t} \rangle$ there exists \nima-model $M_{\mathcal{N}} = \langle W, \mathcal{N}, V_{\mathcal{N}} \rangle$ which is \textit{pointwise equivalent} to $M_{t}$, i.e. $w \Vdash \varphi \Leftrightarrow w \in V_{t} (\varphi)$.
\end{tw}

\begin{proof}

Let us take $M_{t}$ and introduce $\mathcal{N}^{t}_{w}$ for each $w \in W$ just like in Def. \ref{minmax}. We define $V_{\mathcal{N}}: PV \rightarrow P(W)$ in the following way: $V_{\mathcal{N}} = V_{t}$. Now the structure $M_{\mathcal{N}} = \langle W, \mathcal{N}^{t}, V_{\mathcal{N}} \rangle$ is a proper neighborhood model. In fact, we have already shown that it is \nima-frame. By the definition of $V_{t}$ we know that it is monotone in \nima-frame. Let us check pointwise equivalency between both structures.

\item $\rightarrow$

($\Rightarrow$)
Suppose that $w \Vdash \varphi \rightarrow \psi$. Thus $\bigcap \mathcal{N}^{t}_{w} \subseteq \{v \in W; v \nVdash \varphi$ or $v \Vdash \psi\} = -V_{\mathcal{N}}(\varphi) \cup V_{\mathcal{N}}(\psi)$. By induction this last set can be written as $-V_{t}(\varphi) \cup V_{t}(\psi)$. Thus, we can say that $w$ belongs to $\mathscr{Imp}$ defined as in Def. \ref{minmax}. Of course $w \in \bigcap \mathcal{N}^{t}_{w}$. Hence, $w \in V_{t}(\varphi \rightarrow \psi)$.

($\Leftarrow$)
Assume that $w \in V_{t}(\varphi \rightarrow \psi)$. This means that there is at least one point $x \in \mathscr{Imp}$ such that $w \in \bigcap \mathcal{N}^{t}_{w}$. But if $\bigcap \mathcal{N}_{x} \subseteq -V_{t}(\varphi) \cup V_{t}(\psi)$ then we can say that $\bigcap \mathcal{N}_{x} \subseteq -V_{\mathcal{N}}(\varphi) \cup V_{\mathcal{N}} (\psi)$ (by induction). In particular, $w \in -V_{\mathcal{N}}(\varphi) \cup V_{\mathcal{N}} (\psi)$. Thus, in \nima-setting, we have $w \Vdash \varphi \rightarrow \psi$.

\item $\Box$

($\Rightarrow$)
Suppose that $w \Vdash \Box \varphi$. Thus $\bigcup \mathcal{N}_{w} \subseteq V_{\mathcal{N}}(\varphi) = V_{t}(\varphi)$. The last equivalence is a result of induction hypothesis. Now we see that $w \in \mathscr{Mod}$. Of course $w \in \bigcap \mathcal{N}^{t}_{w}$. Then $w \in V_{t}(\Box \varphi)$.

($\Leftarrow$)
Assume that $w \in V_{t}(\Box \varphi)$. Hence, there is at least one world $x \in \mathscr{Mod}$ such that $w \in \bigcap \mathcal{N}^{t}_{x}$. But if $\bigcup \mathcal{N}_{x} \subseteq V_{t}(\varphi)$, then by induction $\bigcup \mathcal{N}_{x} \subseteq V_{\mathcal{N}}(\varphi)$. This means that $x \Vdash \Box \varphi$. By monotonicity of forcing in $\bigcap \mathcal{N}^{t}_{x}$ we can say that $w \Vdash \Box \varphi$.

\end{proof}

\section{Alternative approach}
\label{alt}

Let us define topology in a slightly different way than in Def. \ref{wopen}. Now we shall not work with distinguished sets. On the other hand, we must pay for it by assuming that $\bigcap \mathcal{N}_{w}$ is always contained in each $w$-open set.

\begin{df}

Suppose that we have \nima-frame $M_{\mathcal{N}} = \langle W, \mathcal{N} \rangle$. We say that is $w_{\min}$-open in \nima-structure \textit{iff} $X \subseteq \bigcup \mathcal{N}_{w}$, $\bigcap \mathcal{N}_{w} \subseteq X$ and for every $v \in X$ we have $\bigcap \mathcal{N}_{v} \subseteq X$. We denote $\mathcal{Q}_{w} = \{X \subseteq W; X $ are $w_{\min}$-open $\}$ and call it $w_{\min}$-topology.

\end{df}

\begin{tw}
Assume that we have \nima-frame $F_{\mathcal{N}} = \langle W, \mathcal{N}\rangle$. Then $\langle \bigcup \mathcal{N}_{w}, \mathcal{Q}_{w} \rangle$ is a topological space for every $w \in W$.
\end{tw}

\begin{proof}
It is easy to check conditions of well-defined topology - just as in Th. \ref{topol}. We leave details to the reader.
\end{proof}

\begin{figure}[ht]
\centering
\includegraphics[height=5cm]{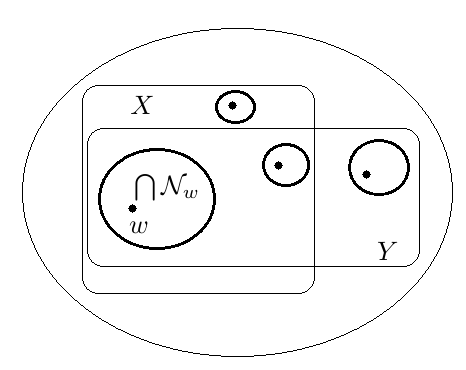}
\caption{Topology ${Q}_{w}$. X, Y are $w$-open.}
\label{fig:obrazek {mtop4}}
\end{figure}

\subsection{From neighborhood frames to multi-topological structures once again}
\label{top2}

Let us introduce the new type of multi-topological structures. In fact, they are \multitopp-frames but with valuation defined in a different way. Recall that $\mathcal{O}^{w}_{\tau}$ denotes the family of all $\tau$-open $w$-neighborhoods and $\text{min} \mathcal{O}^{w}_{\tau}$ is an intersection of such family.

\begin{df}
\label{multitop2}

\multitoppp-model is an ordered triple $M_{t} = \langle W, \mathfrak{W}, V \rangle$ where $\langle W, \mathfrak{W} \rangle$ is a \multitopp-frame and $V_{t}$ is a function from $PV$ into $P(W)$ satisfying the following condition: $V_{t}(q) = \bigcup \mathcal{X}$ where $\mathcal{X} \subseteq \{X \subseteq W$; there is $\langle T, \tau \rangle \in \mathfrak{W}$ and $w \in T$ such that $X = \text{min} \mathcal{O}^{w}_{\tau}\}$.

\end{df}

\begin{df}
\label{defin}
For every \multitoppp-model $M_{t} = \langle W, \mathfrak{W}, V_{t} \rangle$, valuation of formulas is defined as such:

\begin{enumerate}

\item $V_{t}(\varphi \land \psi) = V_{t}(\varphi) \cap V_{t}(\psi)$

\item $V_{t}(\varphi \lor \psi) = V_{t}(\varphi) \cup V_{t}(\psi)$

\item $V_{t}(\varphi \rightarrow \psi) = \bigcup \mathcal{X}$, where $\mathcal{X} = \{X \subseteq W$ such that $X \subseteq -V(\varphi) \cup V(\psi)$ and there are $\langle T, \tau \rangle \in \mathfrak{W}, w \in T$ such that $X = \text{min} \mathcal{O}^{w}_{\tau}\}$.

\item $V(\Box \varphi) = \bigcup \mathcal{X}$, where $\mathcal{X} = \{X \subseteq W$ such that there are $\langle T, \tau \rangle \in \mathfrak{W}, w \in T$ for which $X = \text{min} \mathcal{O}^{w}_{\tau}$ and $T \subseteq V(\varphi)\}$.

\end{enumerate}

\end{df}

We say that formula $\varphi$ is true \textit{iff} in each \multitoppp-model $M_{t} = \langle W, \mathfrak{W}, V_{t} \rangle$ we have $V_{t}(\varphi) = W$.

One can see that in some sense we composed earlier definitions of multi-topological frames, valuations and models. Now our situation is similar to that from section \ref{mtf}. The main difference is that we can work with minimal $\tau$-open sets, i.e. with $\text{min} \mathcal{O}^{w}_{\tau}$.

\begin{tw}
For each \nima-model $M_{\mathcal{N}} = \langle W, \mathcal{N}, V_{\mathcal{N}} \rangle$ there exists \multitoppp-model $M_{t} = \langle W, \mathfrak{W}, V_{t} \rangle$ which is \textit{pointwise equivalent} to $M_{\mathcal{N}}$, i.e. $w \Vdash \varphi \Leftrightarrow w \in V_{t} (\varphi)$.
\end{tw}

\begin{proof}

Assume that we have $M_{\mathcal{N}} = \langle W, \mathcal{N}, V_{\mathcal{N}} \rangle$. Now let us consider the following structure: $M_{t} = \langle W, \mathfrak{W}, V_{t} \rangle$ where:

\begin{enumerate}

\item $\mathfrak{W} = \{ \langle \bigcup \mathcal{N}_{w}, \mathcal{Q}_{w} \rangle; w \in W \}$

\item for each $q \in PV$, $V_{t}(q) = V_{\mathcal{N}}(q)$

\end{enumerate}

It is easy to check that $\langle W, \mathfrak{W} \rangle$ is a well-defined \multitopp-frame. Let us prove pointwise equivalency by means of induction.

\item $\rightarrow$

($\Rightarrow$)

Suppose that $w \Vdash \varphi \rightarrow \psi$. Thus $\bigcap \mathcal{N}_{w} \subseteq \{v \in W; v \nVdash \varphi$ or $v \Vdash \psi\}$. The last set - by induction hypothesis - is equal to $-V_{t}(\varphi) \cup V_{t}(\psi)$. Moreover, $\bigcap \mathcal{N}_{w}$ is an intersection of all $w_{\min}$-open sets (recall Def. \ref{defin}) and $w \in \bigcap \mathcal{N}_{w}$. Thus $w \in V_{t}(\varphi \rightarrow \psi)$.

($\Leftarrow$)

Assume that $w \in V_{t}(\varphi \rightarrow \psi)$. This means two things. First, there is $X \subseteq W$ such that $w \in X$ and $X \subseteq -V_{t}(\varphi) \cup V_{t}(\psi)$. Second, there is (for certain $x \in W$) $\langle \bigcup \mathcal{N}_{x}, \mathcal{Q}_{x} \rangle \in \mathfrak{W}$ such that $X$ is minimal $\mathcal{Q}_{x}$-open $x$-neighborhood. In fact, it means that $X = \bigcap \mathcal{N}_{x}$. So $\bigcap \mathcal{N}_{x} \subseteq -V_{t}(\varphi) \cup V_{t}(\psi) = $[ind. hyp.]$ -V_{\mathcal{N}}(\varphi) \cup V_{\mathcal{N}}(\psi) = \{z \in W; z \nVdash \varphi$ or $z \Vdash \psi\}$. Then, in particular, $x \Vdash \varphi \rightarrow \psi$ and also $w \Vdash \varphi \rightarrow \psi$ (because $w \in \bigcap \mathcal{N}_{x}$ and we have intuitionistic monotonicity of forcing).

\item $\Box$

$\Rightarrow$

Suppose that $w \Vdash \Box \varphi$. Thus $\bigcup \mathcal{N}_{w} \subseteq \{v \in W; v \Vdash \varphi\}$. The last set is - by induction hypothesis - equal to $V_{t}(\varphi)$. We can say that conditions from Def. \ref{defin} are satisfied: our $X$ is $\bigcap \mathcal{N}_{w}$ and our topological space is $\langle \bigcup \mathcal{N}_{w}, \mathcal{Q}_{w} \rangle$. Thus $w \in V_{t}(\Box \varphi)$.

$\Leftarrow$

Assume that $w \in V_{t}(\Box \varphi)$. Thus, we have $X \subseteq W$ such that $w \in X$ and there is (for certain $x \in W$) $\langle \bigcup \mathcal{N}_{x}, \mathcal{Q}_{x} \rangle \in \mathfrak{W}$ for which $X$ is minimal $\mathcal{Q}_{x}$-open $x$-neighborhood (i.e. $\bigcap \mathcal{N}_{x}$) - and $\bigcup \mathcal{N}_{x} \subseteq V_{t}(\varphi)$. By induction hypothesis $\bigcup \mathcal{N}_{x} \subseteq V_{\mathcal{N}}(\varphi)$. Thus, $x \Vdash \Box \varphi$. By monotonicity of forcing, $w \Vdash \varphi$.

\end{proof}

\section{Summary}

In this paper we used a lot of notions and symbols. We have introduced three different concepts of multi-topological frames (models). Moreover, we used the notion of \textit{neighborhood} in three ways. First, we spoke about the class of all neighborhood structures (\nima-frames). Second, we made references to neighborhoods in the topological sense. Third, we used those topological neighborhoods (and other tools) to transform multi-topological frame into certain specific \nima-frame. Hence, we shall repeat the most important things and sum up our considerations.

In section \ref{neigh} we have described neighborhood semantics for intuitionistic modal logic. It is based on the notions of \textit{minimal} ("intuitionistic") and \textit{maximal} ("modal") neighborhoods.

In section \ref{mtf} we have introduced \multitop-frames (models). They are collections of topological spaces. These spaces can intersect or form unions. We assumed that each space $\langle T, \tau \rangle$ has certain distinguished open set $D^{T}$. Then we have shown how it is possible to treat \nima-frames as \multitop-frames. Shortly speaking, the main idea is to make connection between maximal (resp. minimal) neighborhoods and universes $T$ (resp. distinguished sets).

In section \ref{sect} we spoke about \multitopp-frames (models). They are similar to the class of \multitop - but each topology is Alexandroff space and we do not introduce distinguished sets anymore. We have shown how to transform those structures into neighborhood models. Let us repeat main steps of this reasoning. Assume that $W$ is the whole universe of a given \multitopp-frame (i.e. it is set-theoretic sum of universes of all topological spaces of which frame consists). Now let us take an arbitrary $w \in W$. For each topology $\tau$ we have minimal $\tau$-open $w$-neighborhood (because of Alexandroff property). We take intersection of all such minimal neighborhoods and treat it as $\bigcap \mathcal{N}_{w}$ - i.e., as the minimal $w$-neighborhood in the sense of \nima-frames. As for the maximal neighborhood, we take intersection of all topological spaces to which $w$ belongs.

In section \ref{top2} we came back to \nima-frames but we introduced another topology in those structures (different than in section \ref{mtf}). We show that it is possible to transform \nima-models with this topology into \multitoppp-multi-topological models - which are based on \multitopp-frames but with different valuation than in section \ref{sect}.

\end{document}